\documentclass[12pt,a4paper,reqno]{amsart} 
\usepackage{amssymb}
\usepackage{latexsym}
\usepackage{amsmath}
\usepackage{mathrsfs}
\usepackage{amsthm}
\usepackage{xcolor,soul}
\usepackage{calc}
\usepackage[english]{babel}
\usepackage{cite}

\usepackage{xspace}

\newcommand{\rr}[1]{\mathbb R^{#1}}

\newcommand{\nm}[2]{\Vert #1\Vert _{#2}}

\newcommand{\cdo}{\, \cdot \, }

\newcommand{\maclS}{\mathcal S}

\newcommand{\mascF}{\mathscr F}

\newcommand{\mascP}{\mathscr P}
\newcommand{\mascS}{\mathscr S}

\numberwithin{equation}{section}

\newtheorem{thm}{Theorem}
\numberwithin{thm}{section}

\newcommand{\rubrik}{}
\newtheorem{prop}[thm]{Proposition}

\newtheorem{lemma}[thm]{Lemma}
\theoremstyle{definition}
\newtheorem{defn}[thm]{Definition}

\theoremstyle{remark}
\newtheorem{rem}[thm]{Remark}

\author{Nenad Teofanov}

\address{Department of Mathematics and Informatics,
	University of Novi Sad, Serbia}

\email{nenad.teofanov@dmi.uns.ac.rs}

\title{\textbf{Wilson bases and ultradistributions}}

\begin{document}
	


\begin{abstract}
We give a characterization of Gelfand-Shilov type spa\-ces of test functions and their dual spaces of tempered
ultradistributions
by the means of  Wilson bases of exponential decay. We offer two different proofs, and
extend known results to the Roumieu case.
\end{abstract}

\maketitle

\section{Introduction}\label{sec0}

Wilson bases were constructed by I.\ Daubechies, S.\ Jaffard, and J.\
Journ\'e in \cite{DJJ} to overcome constraints arising from the Balian-Low theorem,
and soonafter shown to be unconditional bases for modulation spaces, see \cite{FGW}.
By combining the Wilson bases and tools from time-frequency analysis, approximate diagonalization of
different classes of pseudodifferential operators is obtained in \cite{T1, PT2}.
Wilson bases of Meyer type were used in the study of gravitational waves,
cf. \cite{CJM, NKM}. We refer to \cite{BJLO} for  recent construction of orthonormal Wilson bases in multidimensional case
which overcomes a deficiency of the tensor product construction  used in e.g. \cite{T1,T2}.

\par

Gelfand--Shilov spaces were initially introduced for the analysis  of solutions of certain parabolic initial-value problems  \cite{GS}, and thereafter applied in different contexts when precise estimates of global decay and regularity are needed,
see \cite{Gramchev} for an overview. Recently, Hermite expansions of Fourier transform invariant Gelfand--Shilov spaces, and more generally Pilipovi\'c spaces, are considered in  \cite{Toft18}, see also \cite{P1, langen}. 

\par

In this paper we give a description of Gelfand--Shilov spaces and their dual spaces of tempered ultradistirbutions in terms of Wilson bases. This extends some results from \cite{PT1} given for Beurling type Gelfand--Shilov spaces.
Both Wilson bases and Hermite functions are  orthonormal bases for $ L^2 (\mathbb{R}^d) $
consisting of functions which are well localized in phase--space (time--frequency plane).
From such perspective our results are expected. However, due to the specific structure of Wilson bases, the proofs are based on entirely different arguments then those related to the Hermite basis which utilize recursive relation between Hermite functions and the fact that they are eigenfunctions of the harmonic oscillator.

\par

Instead, we apply the powerful general theory of coorbit spaces,  \cite{FG1, FG2}. The key auxiliary result is  the fact that Wilson bases are unconditional bases for coorbit spaces, \cite{FGW}. We modify and simplify the approach from
\cite{PT1} related to Beurling case, and provide detailed proofs since the more involved Roumieu case contains
nontrivial modifications of arguments given there.
As a consequence of our results, we recover the well known relation between Gelfand--Shilov spaces and modulation spaces.
Furthermore, if that relation is taken as granted, we give an alternative proof of our main results without an explicit reference to coorbit spaces.

\par

Since both proofs are essentially  based on the exponential decay of elements of Wilson bases and asymptotic behavior of the STFT,
the techniques from the present paper can be modified to include other time-frequency representations and
also more general (for example anisotropic)  spaces of test functions and their distribution spaces.
Such investigations are out of scope of the present paper and will be the subject of our future work.

We end the introduction by recalling basic notation which will be used in the sequel.

\subsection*{Notation}
Operators of translations and modulations of a given function $f$ are respectively given by
$ T_x f(\cdot) = f(\cdot - x)$ and $ M_y f(\cdot) = e^{2\pi i y \cdot} f(\cdot)$, $x,y \in \mathbb{R}^d$.
The notation $A\hookrightarrow B$ means that the topological
spaces $A$ and $B$ satisfy $A\subseteq B$ with continuous embeddings.
We write $A(\theta )\lesssim B(\theta )$, $\theta \in \Omega$,
if there is a constant $c>0$ such that $A(\theta )\le cB(\theta )$
for all $\theta \in \Omega$.

\par

The scalar product in $L^2 (\mathbb{R}^d)$ is given by
$$
\langle f, g \rangle = \int _{\mathbb{R}^d} f(x) \overline{g(x)} dx,
$$
and $ \| \cdot \|^2 = \langle \cdot, \cdot \rangle$.

The Fourier transform of an integrable function $ f $ is given by
$$
\hat f (\xi) = {\mathscr F} f (\xi) = \int_{\mathbb{R}^d} e^{-2\pi i \xi x} f(x) dx,  \; \xi \in
\mathbb{R}^d.
$$
It extends uniquely
to a unitary operator on $L^2(\rr d)$.

Let $\phi \in L^2  (\rr d)$ be fixed.
Then the short-time
Fourier transform (STFT) $V_\phi f$ of $f\in L^2(\rr d)$
with respect to the window function $\phi$ is
defined by
\begin{multline}    \label{stft01}
V_\phi f(x,\xi)  =
\mascF (f \, \overline {\phi (\cdo -x)})(\xi )
= (f \cdot T_x \overline{\phi})\hat{} \;  (\xi) \\
 = \int _{\rr d} f(y)\overline {\phi
(y-x)}e^{-2\pi i \xi y}\, dy =
\langle f, M_\xi T_x \phi \rangle, \quad x,\xi \in \rr d.
\end{multline}
Let $f_1,f_2, \phi_1, \phi_2 \in L^2 (\mathbb{R}^d)$. Then $V_{\phi_j } f_j \in L^2 (\mathbb{R}^{2d})$, $j=1,2$, and
it satisfies the following orthogonality relation (\cite[Theorem 3.2.1]{Grbook}):
\begin{equation}
\langle V_{\phi_1}f_1 , V_{\phi_2}f_2 \rangle = \langle \phi_2,\phi _1\rangle  \langle f_1,f_2 \rangle,
\label{orthrel}
\end{equation}
whence $ \| V_{\phi}f \| =  \| \phi \| \cdot \| f\| $. The following fundamental identity of time-frequency analysis
(\cite{Cordero1,Grbook}) is often used:
\begin{equation}
V_\phi f(x, \xi) = e^{-2 \pi i x \cdot \xi} V_{\widehat{\phi}} \widehat{f} (\xi, -x), \;\;\; x,\xi \in \rr d.
\label{fundident}
\end{equation}

By $  \Sigma_1 (\mathbb{R}^d ) $ we denote the Gelfand-Shilov space of smooth functions
given by:
\begin{equation}
f \in  \Sigma_1 (\mathbb{R}^d) \Leftrightarrow
\| f(x)  e^{h\cdot |x|}\|_{L^\infty} < \infty \;
\; \text{and} \;\;
\| \hat f (\omega)  e^{h\cdot |\omega|}\|_{L^\infty}< \infty, \;\; \forall  h > 0,
\label{eqGelfandShilovspace}
\end{equation}
and its dual space is denoted by $ \Sigma_1 ^{\prime} (\mathbb{R}^d ). $

If $ \phi \in \Sigma_1 (\mathbb{R}^d )$, then $\overline {M_\xi T_x \phi} \in  \Sigma_1 (\mathbb{R}^d ), $ so by  \eqref{stft01} it follows that the STFT can be extended to
$ \Sigma_1 ^{\prime} (\mathbb{R}^d ), $ and restricted to $ \Sigma_1 (\mathbb{R}^d )$.

\section{Preliminaries}\label{sec1}

\par

In this section we recall the Wilson  bases, weight functions, coorbit spaces and Gelfand-Shilov type spaces. We also prove some auxiliary results (Lemmas \ref{lema1} and \ref{lema2} and Theorem \ref{esemes}) which will be used in Sections \ref{sec2} and \ref{sec:modulation}.

\subsection{Wilson bases}

Following the idea of K. Wilson \cite{Wilson}, Daubechies, Jaffard and Journe constructed a
 real-valued function $\psi $ such that
\begin{equation}    \label{opadanjepsi}
|\psi (x)| \leq C e^{-a |x|}, \;\;\; |\hat\psi (\xi)| \leq C
e^{-b |\xi|} \;\;\; x,\xi \in \mathbb{R},
\end{equation}
for some constants $ a,b,C > 0 $, and obtain an orthonormal basis (ONB)
$ \{\psi_{l,n} \}_{l \in \mathbb{N}_{0}, n \in
\mathbb{Z}}, $   of $L^2 (\mathbb{R})$, where
\begin{eqnarray}
 \psi_{0,n}  (x) = & T_n \psi (x), & \nonumber \\
 \psi_{l,n} (x)  = & \sqrt 2 \Re (M_l T_{n/2} \psi (x)), &
                      l+n \in 2\mathbb{Z}, l\neq 0,  \label{wilson01} \\
 \psi_{l,n}  (x) = & \sqrt 2 \Im (M_l T_{n/2} \psi(x) ), &
                     l+n \in 2\mathbb{Z} + 1, l\neq 0,   \nonumber
\end{eqnarray}
see \cite{DJJ}. From \eqref{opadanjepsi} it follows that
\begin{equation}    \label{opadanje}
|\psi_{l,n} (x)| \leq C e^{-a |x|}, \;\;\; |\hat\psi_{l,n} (\xi)| \leq C
e^{-b |\xi|}, \;\;\; x,\xi \in \mathbb{R}, (l,n) \in  \mathbb{N}_{0}  \times \mathbb{Z},
\end{equation}
for some constants $ a,b,C > 0 $ depending on $l$ and $n$,
and $ \{\psi_{l,n} \} $ is therefore called the Wilson basis of exponential decay.

Equivalently, \eqref{wilson01} can be  written as
$ \psi_{0,n}   =  T_n \psi $ and
\begin{equation*}
\psi_{l,n} (x) =  \left \{
\begin{array}{rl}
\sqrt{2} \cos 2\pi lx \psi (x-\frac{n}{2}), & l+n \in 2\mathbb{Z}, l\neq 0 , \\
\sqrt{2} \sin 2\pi lx \psi (x-\frac{n}{2}), &  l+n \in 2\mathbb{Z} + 1, l\neq 0.
\end{array}
\right .
\end{equation*}
Moreover, following Gr\"ochenig \cite{Grbook}, we may rewrite \eqref{wilson01} as
\begin{multline}
\psi_{0,n} (x)  =  T_n \psi (x), \qquad
\text{and}  \nonumber \\
 \psi_{l,n} (x)  =  \frac{1}{\sqrt 2}  T_{n/2}  (M_l + (-1)^{n+l} M_{-l})\psi (x),\;\;
                      (l,n) \in \mathbb{N} \times \mathbb{Z}. \label{wilson03}
\end{multline}

To obtain an orthonormal basis  of $L^2 (\mathbb{R}^d)$, Tachizawa in \cite{T1} considered  $d-$dimensional Wilson basis given by the
tensor product
$$ \psi_{l,n} (x) = \psi_{l_1,n_1} (x_1) \otimes
\psi_{l_2,n_2} (x_2) \otimes \dots \otimes \psi_{l_d,n_d} (x_d),
$$ $x= (x_1,x_2,\dots,x_d) \in  \mathbb{R}^d,$ $ l= (l_1,l_2,\dots,l_d)
\in \mathbb{N}_0 ^d, $ $ n= (n_1,n_2,\dots,n_d) $ $ \in \mathbb{Z}^d$.
If $\psi_{l_k,n_k} (x_k)$, $ k =1,\dots, d,$ are Wilson bases of exponential decay, then we have
\begin{equation}      \label{decay-in-d}  \;
|\psi_{l,n} (x)| \leq C e^{-a |x|}, \;\;\; |\hat\psi_{l,n} (\xi)|
\leq C e^{-b |\xi|} \;\;\; x,\xi \in  \mathbb{R}^d , (l,n) \in \mathbb{N}_0 ^d \times
\mathbb{Z}^d,
\end{equation}
for some constants $ a,b,C > 0, $ depending on $l$ and $n.$

The tensor product Wilson bases are  $2^d-$modular, i.e. their elements have $2^d$ peaks in frequency,
which may have undesirable consequences in applications, see \cite{BJLO} for details.
That motivated, Bownik et al.  \cite{BJLO} to
construct a family of orthonormal Wilson bases with $2^k -$modular covering of the frequency domain with $k=1,\dots, d$.
The tensor product Wilson bases turned out to be the special case of their construction.

\par

\subsection{Weight functions}\label{subsec1.1}
A \emph{weight} on $\rr d$ is a positive function $\omega \in  L^\infty _{loc}(\rr d)$
such that $1/\omega \in  L^\infty _{loc}(\rr d)$. The weight $\omega$ on $\rr d$
is called \emph{moderate} if there is a positive locally bounded function
$v$ on $\rr d$ such that
\begin{equation}\label{eq:2}
\omega(x+y)\le C\omega(x)v(y),\quad x,y\in\rr{d},
\end{equation}
for some constant $C\ge 1$. If $\omega$ and $v$ are weights on $\rr d$ such
that \eqref{eq:2} holds, then $\omega$ is also called \emph{$v$-moderate}.
If $v$ can be chosen as polynomial, then $\omega$ is called a weight of polynomial type.
The set of all moderate weights on $\rr d$ is denoted by $\mascP _E(\rr d)$.

\par

The weight $v$ on $\rr d$ is called \emph{submultiplicative},
if it is even and \eqref{eq:2}
holds for $\omega =v$. From now on, $v$ always denotes
a submultiplicative weight if nothing else is stated. In particular,
if \eqref{eq:2} holds and $v$ is submultiplicative, then it follows
by straight-forward computations that
\begin{equation}\label{eq:2Next}
\begin{gathered}
\frac {\omega (x)}{v(y)} \lesssim \omega(x+y) \lesssim \omega(x)v(y),
\\[1ex]
\quad
v(x+y) \lesssim v(x)v(y)
\quad \text{and}\quad v(x)=v(-x),
\quad x,y\in\rr{d}.
\end{gathered}
\end{equation}

\par

If $\omega$ is a moderate weight on $\rr d$,  then there is a
submultiplicative weight
$v$ on $\rr d$ such that \eqref{eq:2} and \eqref{eq:2Next}
hold, see \cite{Groch,To11}. Moreover if $v$ is
submultiplicative on $\rr d$, then
\begin{equation}\label{Eq:CondSubWeights}
1\lesssim v(x) \lesssim e^{r|x|}
\end{equation}
for some constant $r>0$ (cf. \cite{Groch}). In particular, if $\omega$ is moderate, then
\begin{equation}\label{Eq:ModWeightProp}
\omega (x+y)\lesssim \omega (x)e^{r|y|}
\quad \text{and}\quad
e^{-r|x|}\le \omega (x)\lesssim e^{r|x|},\quad
x,y\in \rr d
\end{equation}
for some $r>0$.

\par

We will consider only weight functions $ w $ satisfying Beurling--Domar's
non--quasi\-ana\-ly\-ti\-ci\-ty condition
\begin{equation}\label{bdc}
\sum_{n=1} ^{\infty} n^{-2} \log \omega (nx,ny) < \infty, \;\;\; x,y \in \mathbb R^d.
\end{equation}
The most important examples of  weight functions which satisfy (\ref{bdc})
are $ (1 + |x|)^s, $ $ (1 + |y|)^s, $ $ (1 + |x| + |y|)^s, $  $
 e^{s(|x|^{\gamma} + |y|^{\gamma})} ,  x,y \in  \mathbb R ^d, \; s \geq 0, $ $ \gamma \in (0,1). $

If $\omega \in \mascP _E(\rr d)$ then the weighted $L^2 (\rr d)$ space, $L_\omega ^2 (\rr d)$ is given by
\begin{equation}\label{weightedl2}
f \in L_\omega ^2 (\rr d) \;\;\; \Leftrightarrow
\| f \|_{L_\omega ^2} = \| f \omega \| < \infty.
\end{equation}

\subsection{Coorbit  spaces}\label{subsec1.2}

For the purpose of this paper we focus our attention to subexponential weights of the form
$ \omega_{h,s} (\cdot) = e^{h |\cdot|^{1/s}}, $ $ s > 1,$ $ h\geq 0$.

\begin{defn} \label{def:koorbit}
Let there be given $ s > 1, $ $ h\geq 0 $  and $\phi \in \Sigma_1 (\rr d) \setminus 0. $ The coorbit space
$ {\mathcal C}o Y^{h, s}   (\rr d) $ is defined by
\begin{align}
{\mathcal C}o Y^{h, s}    (\rr d) = \{  f \in & \Sigma_1 ' (\rr d) \;\; | \nonumber \\[1ex]
& \| f \|_{{\mathcal C}o Y^{h, s}  }  \equiv \int_{\rr d}  (
\int_{\rr d} |V_\phi f(x,\xi )|^2 dx ) e^{2h |\xi|^{1/s}} d\xi < \infty
 \}.  \label{koorbit}
\end{align}
\end{defn}

In other words, $ f \in {\mathcal C}o Y^{h, s}    (\rr d) $ if
$  \displaystyle F(\xi)  = \int_{\rr d} |V_\phi f(x,\xi )|^2 dx $ $ \in  L_{\omega_{h,s}} ^2 (\rr d)$
(cf. \eqref{weightedl2}).

This terminology (and notation) is justified by the general theory of coorbit spaces developed in \cite{FG1, FG2}, see also \cite{Dahlke}
for a more recent survey.

From the results given there, it follows that $ {\mathcal C}o Y^{h, s}  (\rr d) $ is a Banach space
invariant under translations, modulations, and complex conjugations.
Moreover, ${\mathcal C}o Y^{h, s}$ is independent on the
choice of $ \phi \in \Sigma_1 (\rr d) \setminus 0 $, see e.g. \cite[Proposition 3.2 {\em ii)}]{Dahlke}.

We will use the following simple results.

\begin{lemma}  \label{lema1}
Let  $ s > 1, $ $ h\geq 0 $,  and $\phi \in \Sigma_1 (\rr d) \setminus 0. $
Then $ \hat f \in {\mathcal C}o Y^{h, s}  (\rr d) $ if and only if
\begin{equation}\label{lema1koorbit}
\int_{\rr d}  (
\int_{\rr d} |V_\phi \hat f(y, \eta)|^2 d\eta ) e^{2h |y|^{1/s}} dy < \infty.
\end{equation}
\end{lemma}

\begin{proof} Since different elements $ \phi \in \Sigma_1 (\rr d) \setminus 0 $ give rise to the same space
$ {\mathcal C}o Y^{h, s}  (\rr d) $ we may take the Gaussian $ \phi (x) = \hat \phi (x) = e^{-\pi x^2} $.
By \eqref{fundident} and the change of variables we have:
\begin{multline*}
\int_{\rr d}  ( \int_{\rr d} |V_\phi \hat f(x,\xi )|^2 dx ) e^{2h |\xi|^{1/s}} d\xi
\\[1ex]
= \int_{\rr d}  ( \int_{\rr d} |V_{\hat \phi} \hat f(x,\xi )|^2 dx ) e^{2h |\xi|^{1/s}} d\xi
\\[1ex]
= \int_{\rr d}  ( \int_{\rr d} |V_{\phi} f (-\xi, x )|^2 dx ) e^{2h |\xi|^{1/s}} d\xi
\\[1ex]
= \int_{\rr d}  ( \int_{\rr d}  |V_{\phi} f (y, \eta )|^2 d\eta ) e^{2h |y|^{1/s}} dy.
\end{multline*}
Thus  $ \hat f \in {\mathcal C}o Y^{h, s}  (\rr d) $ if and only if \eqref{lema1koorbit} holds true.
\end{proof}

We write
\begin{equation}\label{furijekoorbit}
f \in \mathscr{F}{\mathcal C}o Y^{h, s} (\rr d) \quad \text{ if } \quad \hat f \in {\mathcal C}o Y^{h, s} (\rr d).
\end{equation}
(According to the general theory of coorbit spaces it follows that
$\mathscr{F}{\mathcal C}o Y^{h, s} (\rr d)$ is a coorbit space as well.)

\begin{lemma} \label{lema2}
Let there be given $ s > 1 $ and $ h\geq 0 $.
Then we have
\begin{itemize}
\item[a)]
$ f \in {\mathcal C}o Y^{h, s}  (\rr d) $ if and only if $ f(x) e^{h |x|^{1/s}} \in L^2 (\mathbb R ^d).$
\item[b)]
$ f \in \mathscr{F}{\mathcal C}o Y^{h, s}  (\rr d)$ if and only if $ \hat f (\xi) e^{h|\xi|^{1/s}} \in L^2 (\mathbb R ^d).$
\end{itemize}
\end{lemma}

\begin{proof} We again choose $ \phi (x) = e^{-\pi x^2} $ in the definition of ${\mathcal C}o Y^{h, s}  (\rr d)$
and follow the idea of the proof of \cite[Proposition 11.3.1]{Grbook}.
By \eqref{stft01} and the Plancherel theorem we formally have
$$
\int_{\rr d} |V_\phi  f(x,\xi )|^2 d\xi =
\int_{\rr d} |f(t)|^2 |\phi(t-x)|^2 dt.
$$
Since $e^{2 h |x|^{1/s}}$ is a moderate weight,
it follows that
$$
 e^{-2h |u|^{1/s}} e^{2h |t|^{1/s}} \lesssim
e^{2h |t-u|^{1/s}} \lesssim e^{2h |t|^{1/s}} e^{2h |u|^{1/s}}, \;\;\; t,u \in  \mathbb R ^d,
$$
cf. \eqref{eq:2Next}. Therefore,
\begin{multline*}
 \int_{\rr d}   |f(t)|^2 e^{2h |t|^{1/s}} dt
\int_{\rr d} |\phi(u)|^2 e^{-2h |u|^{1/s}} du
\\[1ex]
\lesssim \int_{\rr d} \int_{\rr d} |f(t)|^2 |\phi(u)|^2 e^{2h |t-u|^{1/s}} dt du
\\[1ex]
= \int_{\rr d} \int_{\rr d} |f(t)|^2 |\phi(t-x)|^2 e^{2h |x|^{1/s}} dt dx
\\[1ex]
\lesssim   \int_{\rr d}   |f(t)|^2 e^{2h |t|^{1/s}} dt
\int_{\rr d} |\phi(u)|^2 e^{2h |u|^{1/s}} du,
\end{multline*}
and a) follows.

Part b) follows from a) and the arguments given in the proof of Lemma \ref{lema1}.
\end{proof}

\par

From the general theory of coorbit spaces it follows that  Wilson bases of exponential decay
are unconditional bases for  ${\mathcal C}o Y^{h, s}  (\rr d)$ and  $\mathscr{F}{\mathcal C}o Y^{h, s}  (\rr d)$,
\cite{FGW, PT1}. The precise statement is the following.

\begin{thm} \label{coorbitandwilson}
Let there be given  $ s > 1 $, $ h\geq 0 $,
and the Wilson basis of exponential decay $ \{ \psi_{l,n} \}_{ l \in \mathbb N_{0}, n \in \mathbb Z}$.
Let $ {\mathcal C}o Y^{h, s}  (\rr d)  $ and $\mathscr{F}{\mathcal C}o Y^{h, s}  (\rr d)$  be given by \eqref{koorbit} and
\eqref{furijekoorbit} respectively.
Then we have:
\begin{itemize}
\item[a)] the Wilson basis  $ \{ \psi_{l,n} \}_{ l \in \mathbb N_{0}, n \in \mathbb Z}$ of exponential decay is an unconditional basis for the
coorbit spaces $ {\mathcal C}o Y^{h, s}  (\rr d)  $ and $\mathscr{F}{\mathcal C}o Y^{h, s}  (\rr d)$.
\item[b)] Every function $ f \in {\mathcal C}o Y^{h, s}  (\rr d) $ has the unique
expansion
\begin{equation}\label{wilsonexpansion}
f =
 \sum_{l \in \mathbb  N_{0}, n \in \mathbb Z} c_{l,n} \psi_{l,n}  \;\; \text{where} \;
c_{l,n} =
\langle f,  \psi_{l,n} \rangle, \;\;\;
l \in \mathbb{N}_{0}, n \in \mathbb{Z},
\end{equation}
and
\begin{equation}\label{rastkoef1}
\sum\limits_{l = 0} ^{\infty} \left ( \sum_{n\in \mathbb Z} |c_{l,n}|^2 \right )
e^{2h|l|^{1/s}} < \infty.
\end{equation}
\item[c)]  Every function $ f \in \mathscr{F}{\mathcal C}o Y^{h, s}  (\rr d) $ has the unique
expansion of the form \eqref{wilsonexpansion} and
\begin{equation}\label{rastkoef2}
\sum\limits_{l = 0} ^{\infty} \left ( \sum_{n\in \mathbb Z} |c_{l,n}|^2 \right )
e^{2h|\frac{n}{2}|^{1/s}} < \infty.
\end{equation}
\end{itemize}
\end{thm}

\begin{proof} The proof is omitted since it follows by the arguments given in
the proof of Theorem 4 in \cite{FGW}, where polynomial type weights are considered instead.
The subexponential type weights considered here are treated in \cite{PT1}, see Theorem 4.4 and Remark 4.5 given there.
\end{proof}

Theorem  \ref{coorbitandwilson} is the main auxiliary result which
will be used to prove representation theorem for Gelfand-Shilov spaces, Theorem \ref{main1} a).

\subsection{Gelfand-Shilov spaces}\label{subsec1.2}
Gelfand and Shilov introduced the spa\-ces of type $S$, for the analysis  of solutions of certain parabolic initial-value problems.
A comprehensive study of those spaces   which are afterwards called Gelfand-Shilov spaces is given in \cite{GS}.
We focus our attention to the case when regularity and  decay are controlled by the so called Gevrey sequences
$M_p = p!^s, $ when $s > 0 $, and refer to e.g. \cite{Te3} for an overview of a more general situation.

Let $0<s $ be fixed. Then the (Fourier invariant)
Gelfand-Shilov space $\maclS _s(\rr d)$ ($\Sigma _s(\mathbb R^ d)$) of
Roumieu type (Beurling type) consists of all $f\in C^\infty (\mathbb R ^d)$
such that
\begin{equation}\label{gfseminorm}
\nm f{\maclS _{s,h}}\equiv \sup_{\alpha ,\beta \in \mathbb N^d} \frac {|x^\alpha \partial ^\beta
f(x)|}{h^{|\alpha  + \beta |}(\alpha !\, \beta !)^s}
\end{equation}
is finite for some $h>0$ (for every $h>0$).  The semi-norms
$\nm \cdo {\maclS _{s,h}}$ induce inductive limit topology for the
space $\maclS _s(\rr d)$, and projective limit topology for $\Sigma _s(\rr d)$. Thus the
former space becomes an LS space, while the latter space is an FS space (Fr{\'e}chet-Schwartz space) under these topologies.

The space $\maclS _s(\rr d)\neq \{ 0\}$ ($\Sigma _s(\rr d)\neq \{0\}$), if and only if
$s\ge \frac 12$ ($s> \frac 12$).

\par

The \emph{Gelfand--Shilov distribution spaces} $\maclS _s'(\rr d)$
and $\Sigma _s'(\rr d)$ (also known as spaces of tempered ultradistributions)
are the dual spaces of $\maclS _s(\rr d)$
and $\Sigma _s(\rr d)$, respectively.

\par

We have
\begin{equation}\label{GSembeddings}
\begin{aligned}
\maclS _{1/2} (\rr d) &\hookrightarrow \Sigma _s  (\rr d) \hookrightarrow
\maclS _s (\rr d)
\hookrightarrow  \Sigma _t(\rr d)
\\[1ex]
&\hookrightarrow
\mascS (\rr d)
\hookrightarrow \mascS '(\rr d) 
\hookrightarrow \Sigma _t' (\rr d)
\\[1ex]
&\hookrightarrow  \maclS _s'(\rr d)
\hookrightarrow  \Sigma _s'(\rr d) \hookrightarrow \maclS _{1/2} '(\rr d),
\quad \frac 12<s<t.
\end{aligned}
\end{equation}

\par

The Fourier transform $\mathscr F$ extends
uniquely to homeomorphisms on $\maclS _s'(\rr d)$ and on $\Sigma _s'(\rr d)$. Furthermore,
$\mascF$ restricts to
homeomorphisms on $\maclS _s(\rr d)$ and on $\Sigma _s (\rr d)$.
Similar facts hold true
when the Fourier transform is replaced by a partial
Fourier transform.

Fourier transform invariance of
$\maclS _s(\rr d)$ and $\Sigma _s (\rr d)$ follows from the following result which also gives
a characterization of Gelfand-Shilov spaces in terms of coorbit spaces.

\par

\begin{thm} \label{esemes}  Let there be given $ s>1.$
The following conditions are equivalent:
\begin{itemize}
\item[a)] $ f \in {\mathcal S}_s (\rr d) $ ($ f \in \Sigma_s (\rr d) $);
\item[b)] there exists $h>0$ (for every  $h>0$)
\begin{equation}\label{GSsymmetric}
\sup_{x \in\rr d}  | f(x) | e^{h|x|^{1/s}} < \infty \quad \text{and }\quad
\displaystyle \sup_{\xi \in \rr d}  | \hat f (\xi) | e^{h|\xi|^{1/s}} < \infty ;
\end{equation}
\item[c)] there exists $h>0$ (for every  $h>0$) such that
$$
f \in  {\mathcal C}o Y^{h, s}  (\rr d)   \cap \mathscr{F}{\mathcal C}o Y^{h, s}  (\rr d).
$$
\end{itemize}
\end{thm}

\begin{proof}
a) $ \Leftrightarrow$  b) is  well known, and holds for all $ s>0 $, \cite{CCK, GZ, KPP}.

 b) $ \Leftrightarrow$  c). By $d-$dimensional version of \cite[Theorem 2.2]{PT1}, it follows that
the sup-norms in \eqref{GSsymmetric} can be replaced by $ L^2 _{\omega_{k,s}} $--norms,
i.e. for every $k>0$ we have
\begin{equation}\label{GSsymmetricL2}
\int_{\rr d} |f(x)|^2  e^{2k |x|^{1/s}} dx < \infty \quad \text{and} \quad
\int_{\rr d} |\hat f(\xi)|^2  e^{2k |\xi|^{1/s}} d\xi < \infty.
\end{equation}
This, together with Lemma \ref{lema2}
gives
$$
f \in \Sigma_s (\rr d) \quad \Leftrightarrow \quad
f \in  {\mathcal C}o Y^{h, s}  (\rr d)   \cap \mathscr{F}{\mathcal C}o Y^{h, s}  (\rr d)
$$
for every $h>0$.

For the Roumieu case, by slight modifications of the proof of  \cite[Theorem 2.2]{PT1}, it also follows that
\eqref{GSsymmetric} holds for some $h>0$ if and only if
\eqref{GSsymmetricL2} holds for some $k>0$. Again, Lemma \ref{lema2}
imply that this is equivalent with
$ \displaystyle f \in  {\mathcal C}o Y^{h, s}  (\rr d)   \cap \mathscr{F}{\mathcal C}o Y^{h, s}  (\rr d)
$  for some $h>0$, and the proof is finished.
\end{proof}

\par

Note that, since the equivalence between a) and b) in Theorem \ref{esemes} holds even if $ s= 1$,
if $\psi $ satisfies \eqref{opadanjepsi}, then the Wilson basis elements $ \psi_{l,n}, $ $l \in \mathbb{N}_{0}$, 
$n \in \mathbb{Z}, $ given by \eqref{wilson01} belong to $ {\mathcal S}_1 (\rr d) $, cf. \eqref{opadanje}.

\par

The restriction $ s>1 $ when proving b) $ \Leftrightarrow $ c) in Theorem \ref{esemes}  comes from Definition \ref{def:koorbit}. In fact, in the general theory of coorbit spaces, as presented in \cite{FG1, FG2}, an important role  is played by BUPUs
({\em bounded uniform partitions of unity}) consisting of compactly supported smooth functions. In such setting, coorbit spaces consist of non--quasianalytic functions.

\par

From definitions of $ {\mathcal C}o Y^{h, s}  (\rr d)  $, $ \mathscr{F}{\mathcal C}o Y^{h, s}  (\rr d)$
and Theorem \ref{esemes} c) it follows  that
Gelfand-Shilov spaces are essentially characterized by the decay estimates of the short-time Fourier
transform. Note that the estimates given in Proposition \ref{stftGelfand2} below  employ the sup-norm ($L^\infty$-norm)
whereas in Theorem \ref{esemes} c) the $L^2$-norm related to $ {\mathcal C}o Y^{h, s}  (\rr d)  $ and
$\mathscr{F}{\mathcal C}o Y^{h, s}  (\rr d) $ is considered instead. (In fact, any $L^p-$norm ($1\leq p\leq \infty$) can be
 used, see \cite{P1}.)

\par

\begin{prop}\label{stftGelfand2}
Let $s\ge \frac 12$ ($s>\frac 12$), $\phi \in \maclS _s(\rr d)\setminus 0$
($\phi \in \Sigma _s(\rr d)\setminus 0$) and let $f$ be a
Gelfand-Shilov distribution on $\rr d$. Then the following is true:
\begin{itemize}
\item[a)] $f\in \maclS _s (\rr d)$ ($f\in \Sigma_s(\rr d)$), if and only if
\begin{equation}\label{stftexpest2}
|V_\phi f(x,\xi )| \lesssim  e^{-r (|x|^{\frac 1s}+|\xi |^{\frac 1s})}, \quad x,\xi \in \rr d,
\end{equation}
for some $r > 0$ (for every $r>0$).
\item[b)] $f\in \maclS _s'(\rr d)$ ($f\in \Sigma _s'(\rr d)$), if and only if
\begin{equation}\label{stftexpest2Dist}
|V_\phi f(x,\xi )| \lesssim  e^{r(|x|^{\frac 1s}+|\xi |^{\frac 1\sigma})}, \quad
x,\xi \in \rr d,
\end{equation}
for every $r > 0$ (for some $r > 0$).
\end{itemize}
\end{prop}

We omit the proof since the first part follows from \cite[Theorem 2.7]{GZ}
and the second part from \cite[Proposition 2.2]{Toft18}.
See also \cite{CPRT10} for related results.

\par

From these investigations and by \cite[Theorem 2.3]{To11} it follows that the definition of the map
$(f,\phi)\mapsto V_{\phi} f$ from $L^2 (\rr d) \times L^2 S (\rr d)$
to $ L^2 (\rr {2d})$ is uniquely extendable to a continuous map from
$\maclS _s'(\rr d)\times \maclS_s'(\rr d)$
to $\maclS_s'(\rr {2d})$, and restricts to a continuous map
from $\maclS _s (\rr d)\times \maclS _s (\rr d)$
to $\maclS _s(\rr {2d})$.
The same conclusion holds with $\Sigma _s$ in place of
$\maclS_s$, at each place. Therefore, Definition \ref{def:koorbit} can be appropriately
modified to include ultradistributions $  f \in  \maclS_s'(\rr {d})$
(or $f \in \Sigma _s ' (\rr {d})$, and we will use such extension from now on.

\section{Main results} \label{sec2}

In this section we discuss Wilson bases expansions in the context of Gelfand-Shilov spaces and their dual spaces of tempered ultradistributions.

\begin{thm} \label{main1}
Let $ s>1$ and let there be given a Wilson basis of exponential decay $ \{ \psi_{l,n} \}_{ l \in
\mathbb{N}_{0}, n \in \mathbb{Z}}. $

\begin{itemize}
\item[a)]
If  $f \in \maclS _s (\rr d) $ ($f \in \Sigma _s (\rr d) $) then
\begin{equation} \label{expansion}
f = \sum\limits_{ l \in \mathbb{N}_{0}} \sum\limits_{n \in \mathbb{Z} } \langle f,  \psi_{l,n} \rangle \psi_{l,n},
\end{equation}
with the unconditional convergence in $ \maclS _s (\rr d) $ (in $ \Sigma _s (\rr d) $)
and
$$
 \sum\limits_{ l \in \mathbb{N}_{0}, n \in \mathbb{Z} } |c_{l,n}|^2
e^{2k(|n/2| + |l|)^{1/s}} < \infty, \;\; \mbox{ for some }
\;\; \text{(for all)} \; k \geq 0,
$$
where $ c_{l,n} = \langle f,  \psi_{l,n} \rangle, $ $l \in
\mathbb{N}_{0}$, $n \in \mathbb{Z}. $

\item[b)]
Conversely, if $ (c_{l,n})_{l \in \mathbb{N}_{0}, n \in \mathbb{Z}} $ is a (double)
sequence  such that
\begin{equation} \label{mainthm1b}
 \sum\limits_{ l \in \mathbb{N}_{0}, n \in \mathbb{Z} } |c_{l,n}|^2
e^{2k(|n/2| + |l|)^{1/s}} < \infty,
\end{equation}
 for some  (for all) $ k \geq 0$, then
there exists a function $ f \in  \maclS _s (\rr d) $ ($f \in \Sigma _s (\rr d) $) such that \eqref{expansion} holds with
$ c_{l,n} = \langle f,  \psi_{l,n} \rangle, $ $l \in \mathbb{N}_{0}$, $n \in \mathbb{Z}. $
\end{itemize}
\end{thm}

\begin{proof}
a) We prove the Roumieu case, since the Beurling case is given in \cite[Theorem 5.1 a)]{PT1}. Let  $f \in \maclS _s (\rr d) $.
By Theorem \ref{esemes} c) we have that
$$
f \in  {\mathcal C}o Y^{h, s}  (\rr d)   \cap \mathscr{F}{\mathcal C}o Y^{h, s}  (\rr d)
$$
for some $h>0$. Then, from Theorem \ref{coorbitandwilson} it follows that
\eqref{rastkoef1} and \eqref{rastkoef2} hold for that constant $h>0.$
Therefore,
\begin{multline*}
\sum\limits_{ l \in \mathbb{N}_{0}, n \in \mathbb{Z} } |c_{l,n}|^2 e^{2\frac{h}{2}(|n/2| + |l|)^{1/s}} \\
\leq
\sum\limits_{ l \in \mathbb{N}_{0}, n \in \mathbb{Z} } |c_{l,n}| e^{2\frac{h}{2}|n/2|^{1/s}}
\cdot |c_{l,n}| e^{2\frac{h}{2}|l|^{1/s}} \\
\leq
(\sum\limits_{ l \in \mathbb{N}_{0}, n \in \mathbb{Z} } |c_{l,n}|^2 e^{2h|n/2|^{1/s}} )^{1/2}
\cdot (\sum\limits_{l \in \mathbb{N}_{0}, n \in \mathbb{Z} } |c_{l,n}|^2 e^{2h|l|^{1/s}} )^{1/2}
 <  \infty,
\end{multline*}
so that
$$
 \sum\limits_{ l \in \mathbb{N}_{0}, n \in \mathbb{Z} } |c_{l,n}|^2
e^{2k(|n/2| + |l|)^{1/s}} < \infty
$$
for $k = h/2.$ The unconditional convergence follows from Theorem \ref{coorbitandwilson} a).

To prove  b), we note that \eqref{mainthm1b} obviously implies
$$
 \sum\limits_{ l \in \mathbb{N}_{0}, n \in \mathbb{Z} } |c_{l,n}|^2
e^{2k|n/2| ^{1/s}} < \infty,
$$
and
$$
 \sum\limits_{ l \in \mathbb{N}_{0}, n \in \mathbb{Z} } |c_{l,n}|^2
e^{2k |l|^{1/s}} < \infty,
$$
so that
$$
f = \sum\limits_{ l \in \mathbb{N}_{0}} \sum\limits_{n \in \mathbb{Z} } c_{l,n} \psi_{l,n} \in
 {\mathcal C}o Y^{h, s}  (\rr d)   \cap \mathscr{F}{\mathcal C}o Y^{h, s}  (\rr d),
$$
and since the Wilson basis is an ONB we have that $  c_{l,n} = \langle f,  \psi_{l,n} \rangle $,
$l \in \mathbb{N}_{0}$, $ n \in \mathbb{Z}$. Now, by Theorem \ref{esemes} we conclude that $ f \in  \maclS _s (\rr d) $, and the proof is finished.
\end{proof}

For the proof of Theorem  \ref{main2} we need a simple lemma on divergent series. We note that a  similar argument is used in the proof of  \cite[Theorem 9.6-1]{Zem}. To be self-contained we provide the proof in Appendix.

\begin{lemma} \label{seriesestimate}
Let  $(a_n)_{n \in \mathbb{N}_0} $ be a zero convergent sequence of non-negative numbers such that
$$
\sum_{n \in \mathbb{N}_0} a_n = + \infty.
$$
Then there exists an increasing sequence of integers $ m_l, $ $ l \in  \mathbb{N},$ such that
$$
1 < \sum_{n = m_{l-1}} ^{m_l - 1} a_n  < 3.
$$
\end{lemma}

\begin{thm} \label{main2}
Let $ s>1$ and let there be given a Wilson basis of exponential decay $ \{ \psi_{l,n} \}_{ l \in
\mathbb{N}_{0}, n \in \mathbb{Z}}. $

\begin{itemize}
\item[a)]
Every $f \in \maclS _s ' (\rr d) $ ($f \in \Sigma _s '(\rr d) $)  has a unique expansion
$$
\displaystyle f = \sum\limits_{ l \in \mathbb{N}_{0}, n
\in \mathbb{Z} } \langle f, \psi_{l,n} \rangle  \psi_{l,n}
$$
in $ \maclS _s ' (\rr d) $ (in $ \Sigma _s '(\rr d) $) and
\begin{equation} \label{fseries}
 \sum\limits_{ l \in \mathbb{N}_{0}, n \in \mathbb{Z} } |c_{l,n}|^2
  e^{-2h(|n/2| + |l|)^{1/s}} < \infty,
\end{equation}
 for every  (for some) $h \geq 0$,
where
$ c_{l,n} = \langle f,  \psi_{l,n} \rangle, $ $ l \in \mathbb{N}_{0}, n \in \mathbb{Z}.$

\item[b)] Conversely, if (\ref{fseries}) holds for some sequence $ (c_{l,n})_{l
\in \mathbb{N}_{0}, n \in \mathbb{Z}} $ and for every  (for some) $h \geq 0$,
then there exists
$ f\in \maclS _s ' (\rr d) $ ($f \in \Sigma _s '(\rr d) $) such that
$$ \displaystyle f = \sum\limits_{ l
\in \mathbb{N}_{0}, n \in \mathbb{Z} } c_{l,n}  \psi_{l,n}
$$
in  $\maclS _s ' (\rr d) $ (in $ \Sigma _s '(\rr d) $).
\end{itemize}
\end{thm}

\begin{proof}
The  Beurling case can be proved by making appropriate changes if the proof of \cite[Theorem 9.6-1]{Zem}, cf. \cite{Teofthesis}.
However, since the proof for  the Roumieu case contains nontrivial modifications of Zemanian's  proof, we
provide it  here.

b)
Let \eqref{fseries} holds for some $ h \geq 0, $ and let
$ \displaystyle f = \sum\limits_{ l \in \mathbb{N}_{0}, n \in \mathbb{Z} } c_{l,n}  \psi_{l,n}.$

If $ \phi \in \maclS _s  (\rr d) $ then we have
\begin{multline*}
| \langle f, \phi \rangle | =
| \langle  \sum_{  l \in \mathbb{N}_{0}, n \in \mathbb{Z}  } c_{l,n} \psi_{l,n}, \phi \rangle |
\\
\leq
| \langle  \sum_{  l \in \mathbb{N}_{0}, n \in \mathbb{Z}} c_{l,n}
e^{-h(|\frac{n}{2}| + |l|)^{1/s}} e^{h(|\frac{n}{2}| + |l|)^{1/s}}
\psi_{l,n}, \phi \rangle |
\\
 \leq
  \sum_{ l \in \mathbb{N}_{0}, n \in \mathbb{Z}} | c_{l,n} |
  e^{-h(|\frac{n}{2}| + |l|)^{1/s}} e^{h(|\frac{n}{2}| + |l|)^{1/s}}
  | \langle \psi_{l,n}, \phi \rangle |
\\
 \leq
  \big ( \sum_{ l \in \mathbb{N}_{0}, n \in \mathbb{Z} } | c_{l,n} |^2
  e^{-2h(|\frac{n}{2}| + |l|)^{1/s}} \big )^{\frac{1}{2}}
  \big ( \sum_{ l \in \mathbb{N}_{0}, n \in \mathbb{Z} } |\langle \psi_{l,n}, \phi \rangle |^2 e^{2h(|\frac{n}{2}| + |l|)^{1/s}}
  \big )^{\frac{1}{2}}.
\end{multline*}

Since $ \phi \in \maclS _s  (\rr d) $, by Theorem \ref{main1} a) it follows that we can choose $h\geq0$ such that
$$
\sum_{ l \in \mathbb{N}_{0}, n \in \mathbb{Z} } |\langle \psi_{l,n}, \phi \rangle |^2 e^{2h(|\frac{n}{2}| + |l|)^{1/s}} <\infty,
$$
and for such choice of  $h\geq 0$, by (\ref{fseries}) it follows that
$$
\sum_{ l \in \mathbb{N}_{0}, n \in \mathbb{Z} } | c_{l,n} |^2
  e^{-2h(|\frac{n}{2}| + |l|)^{1/s}} < \infty,
$$
and we conclude that $  | \langle   f, \phi \rangle |  <  \infty, $ so that $ f\in \maclS _s ' (\rr d) $,
and b) is proved.

Before we proceed we note that (\ref{fseries})
implies that $ \displaystyle |c_{l,n}| \leq C e^{h(|\frac{n}{2}| + |l|)^{1/s}} $
for every $  h \geq 0, $ and all $ l \in \mathbb{N}_{0}$, $n \in \mathbb{Z}$, that is,
the (double indexed) sequence
$ (|c_{l,n}|  e^{-h(|\frac{n}{2}| + |l|)^{1/s}} )_{ l \in \mathbb{N}_{0}, n \in \mathbb{Z} } $ is bounded  for every $ h \geq 0.$

a)  Let $f \in \maclS _s ' (\rr d) $ and consider
$ \displaystyle  \sum\limits_{ l \in \mathbb{N}_{0}, n \in \mathbb{Z} } c_{l,n}  \psi_{l,n}$
with $ c_{l,n}=\langle f, \psi_{l,n} \rangle,$ $l \in \mathbb{N}_{0}$, $n \in \mathbb{Z} $.
If $ \phi \in \maclS _s  (\rr d) $ then by \eqref{expansion} we have
\begin{multline*}
| \langle  \sum_{  l \in \mathbb{N}_{0}, n \in \mathbb{Z}  } c_{l,n} \psi_{l,n}, \phi \rangle |
=
| \sum_{  l \in \mathbb{N}_{0}, n \in \mathbb{Z}} c_{l,n}
\langle  \psi_{l,n}, \phi \rangle |
\\
=
| \sum_{  l \in \mathbb{N}_{0}, n \in \mathbb{Z}} \langle f,
\langle \phi,  \psi_{l,n}, \rangle  \psi_{l,n} \rangle |
=
|\langle f,  \sum_{  l \in \mathbb{N}_{0}, n \in \mathbb{Z}}
\langle \phi,  \psi_{l,n}, \rangle  \psi_{l,n} \rangle |
\\
= | \langle f, \phi \rangle | < \infty,
\end{multline*}
so that the expansion is unique and
$$
\langle  \sum_{  l \in \mathbb{N}_{0}, n \in \mathbb{Z}  }\langle f, \psi_{l,n} \rangle \psi_{l,n}, \sum_{  k \in \mathbb{N}_{0}, m \in \mathbb{Z}  }
\langle \phi, \psi_{k,m} \rangle \psi_{k,m} \rangle
= \sum_{  l \in \mathbb{N}_{0}, n \in \mathbb{Z}  } \overline{a_{l,n}} c_{l,n} < \infty,
$$
where $a_{l,n} = \langle \phi, \psi_{l,n} \rangle,$  $l \in \mathbb{N}_{0}$, $n \in \mathbb{Z} $.

Next we prove that  the sequence
$ ( e^{-k(|\frac{n}{2}| + |l|)^{1/s}} c_{l,n} )_{ l \in \mathbb{N}_{0}, n
\in \mathbb{Z} } $ is bounded for every $ k> 0. $

We give the proof by contradiction:
suppose that there exists $ k_0 >0$ such that the sequence
$  ( e^{-k_0(|\frac{n}{2}| + |l|)^{1/s}} c_{l,n} )_{ l \in
\mathbb{N}{0}, n \in \mathbb{Z} }  $ is unbounded. Then there
exists a sequence of increasing (by components) indices $ (l_m, n_m)_{ m \in \mathbb{N}} $ such that
$$
e^{-k_0 ( |l_m|+ |\frac{n_m}{2}| )^{1/s}} | c_{l_m,n_m} | \geq m \;\;\; m \in \mathbb{N}.
$$
Next we consider the sequence $ (a_{l,n})$ with the following properties
\begin{enumerate}
\item $\overline{a_{l,n}} c_{l,n} = |a_{l,n} c_{l,n}| $,
\item $ \displaystyle | a_{l_m, n_m}|  = e^{-k_0 ( |l_m|+ |\frac{n_m}{2}| )^{1/s}} \cdot \frac{1}{m}$,
\item $ a_{l,n} = 0 $ when $ (l,n) \neq (l_m, n_m).$
\end{enumerate}
This gives
\begin{multline*}
 \sum_{  l \in \mathbb{N}_{0}, n \in \mathbb{Z}  } |e^{k_0(|\frac{n}{2}| + |l|)^{1/s}}  a_{l,n} |^2
\\
=\sum_{  m\in \mathbb{N} } e^{2 k_0(|\frac{n_m}{2}| + |l_m|)^{1/s}} e^{-2k_0 ( |l_m|+ |\frac{n_m}{2}| )^{1/s}} \cdot \frac{1}{m^2} =\sum_{  m\in \mathbb{N} } \frac{1}{m^2} < \infty.
\end{multline*}
By Theorem \ref{main1} b) it follows that
$ \phi = \sum_{  l \in \mathbb{N}_{0}, n \in \mathbb{Z}  }  a_{l,n}\psi_{l,n} \in  \maclS _s (\rr d) $,  so that
$\sum_{  l \in \mathbb{N}_{0}, n \in \mathbb{Z}  } \overline{a_{l,n}} c_{l,n} < \infty.$ On the other hand,
\begin{multline*}
\sum_{  l \in \mathbb{N}_{0}, n \in \mathbb{Z}  } \overline{a_{l,n}} c_{l,n} = \sum_{  l \in \mathbb{N}_{0}, n \in \mathbb{Z}  } |a_{l,n} c_{l,n}|  \\
= \sum_{  m\in \mathbb{N}_{0} }  e^{-k_0 ( |l_m|+ |\frac{n_m}{2}| )^{1/s}} \cdot \frac{1}{m} |c_{l,n}|
\geq \sum_{  m\in \mathbb{N}_{0} } 1 = \infty,
\end{multline*}
which gives the contradiction.

\par

Thus, we conclude that the sequence
$  ( e^{-k(|\frac{n}{2}| + |l|)^{1/s}} c_{l,n} )_{ l \in
\mathbb{N}_{0}, n \in \mathbb{Z} }  $ is bounded for every $k>0$.

Finally, we prove that \eqref{fseries} holds for every $ h>0$. Again we give the proof by contradiction.
Suppose that there exists $ h_0 > 0 $ such that
\begin{equation} \label{eq:divergence}
\sum_{  l \in \mathbb{N}_{0}, n \in \mathbb{Z}  } b_{l,n} =
\sum_{  l \in \mathbb{N}_{0}, n \in \mathbb{Z}  } |c_{l,n} |^2  e^{-2h_0(|\frac{n}{2}| + |l|)^{1/s}}  = \infty.
\end{equation}
Since $  ( e^{-h_0 (|\frac{n}{2}| + |l|)^{1/s}} c_{l,n} )_{ l \in
\mathbb{N}_{0}, n \in \mathbb{Z} }  $ is bounded, it follows that $ (b_{l,n}) $ is a zero convergent sequence. By
Lemma \ref{seriesestimate} it follows  that there is
an increasing sequence of  indices $ (l_m, n_m)_{ m \in \mathbb{N}} $  such that
\begin{equation} \label{eq:bounds}
1 < \sum_{j=l_{m-1}} ^{l_m -1} \sum_{k= n_{m-1}} ^{n_m - 1}  |c_{j,k} |^2  e^{-2h_0(|\frac{k}{2}| + |j|)^{1/s}}   < 3.
\end{equation}

By choosing
$$
a_{j,k}  = c_{j,k} e^{-2 h_0 (|\frac{k}{2}| + |j|)^{1/s}} \frac{1}{m}, \qquad j=l_{m-1}, \dots, l_m -1, \quad
k= n_{m-1}, \dots, n_m - 1,
$$
we obtain
\begin{multline*}
\sum_{j=l_{m-1}} ^{l_m -1} \sum_{k= n_{m-1}} ^{n_m - 1}  |a_{j,k}|^2  e^{2 h_0 (|\frac{k}{2}| + |j|)^{1/s}} \\
= \sum_{j=l_{m-1}} ^{l_m -1} \sum_{k= n_{m-1}} ^{n_m - 1}  |c_{j,k}|^2  e^{-2 h_0 (|\frac{k}{2}| + |j|)^{1/s}} \frac{1}{m^2} \
<  \frac{3}{m^2},
\end{multline*}
for every $m \in  \mathbb{N}$, where we used \eqref{eq:bounds}. Thus,
\begin{multline*}
\sum_{ l \in \mathbb{N}_{0}, n \in \mathbb{Z}}  |a_{l,n}|^2  e^{2 h_0 (|\frac{n}{2}| + |l|)^{1/s}} \\
= \sum_{j< l_{0}} \sum_{k < n_{0}}  |a_{l,n}|^2  e^{2 h_0 (|\frac{n}{2}| + |l|)^{1/s}}  +
\sum_{j\geq l_{0}} \sum_{k \geq n_{0}}  |a_{l,n}|^2  e^{2 h_0 (|\frac{n}{2}| + |l|)^{1/s}}
\\
< C + \sum_{m \in  \mathbb{N}} \frac{3}{m^2} < \infty.
\end{multline*}

By Theorem \ref{main1} b) it follows that $\sum_{  l \in \mathbb{N}_{0}, n \in \mathbb{Z}  }  a_{l,n}\psi_{l,n} \in  \maclS _s (\rr d) $,
and therefore
\begin{equation} \label{eq:finiteseries}
\sum_{  l \in \mathbb{N}_{0}, n \in \mathbb{Z}  } \overline{a_{l,n}} c_{l,n} < \infty.
\end{equation}
However, by using the left hand side inequality in \eqref{eq:bounds} we obtain
\begin{multline*}
\sum_{j=l_{m-1}} ^{l_m -1} \sum_{k= n_{m-1}} ^{n_m - 1}  |a_{j,k} c_{j,k} | \\
= \sum_{j=l_{m-1}} ^{l_m -1} \sum_{k= n_{m-1}} ^{n_m - 1}  |c_{j,k}|^2  e^{-2 h_0 (|\frac{k}{2}| + |j|)^{1/s}} \frac{1}{m}
\geq   \frac{1}{m},
\end{multline*}
for each $ m \in  \mathbb{N}$, so that
$$
\sum_{  l \in \mathbb{N}_{0}, n \in \mathbb{Z}  } \overline{a_{l,n}} c_{l,n} =
\sum_{  l \in \mathbb{N}_{0}, n \in \mathbb{Z}  } |a_{l,n} c_{l,n} | \\
\geq \sum_{  l \geq l_{0}, n  \geq n_{0} } |a_{l,n} c_{l,n} |  \geq \sum_{m \in  \mathbb{N}}   \frac{1}{m} = \infty.
$$
This is a contradiction with \eqref{eq:finiteseries}. We conclude that the assumption \eqref{eq:divergence} can not hold.
Therefore
$$
\sum_{  l \in \mathbb{N}_{0}, n \in \mathbb{Z}  } |c_{l,n} |^2  e^{-2h(|\frac{n}{2}| + |l|)^{1/s}}  < \infty
$$
for every $h>0$ which completes the proof.

\end{proof}

\section{Alternative proof via modulation spaces}\label{sec:modulation}

Modulation spaces, originally introduced by Feichtinger in \cite{F1},
are recognized as appropriate family of spaces when dealing with problems of time-frequency analysis, see
\cite{F1,Grbook,Benyi,Cordero1}, to mention just a few references. A broader family of
modulation spaces, including quasi-Banach spaces when the Lebesgue parameters $p,q$ belong to $(0,1)$
is studied in e.g. \cite{To25}.

\par

Let there be given $\phi \in \Sigma _1 (\rr d)\setminus 0$, $p,q\in [1,\infty ]$
and $\omega \in\mascP _E(\rr {2d})$. Then the
\emph{modulation space} $M^{p,q}_{\omega }(\rr d)$ consists of all
Gelfand-Shilov distributions $f \in \Sigma _1 ' (\rr d)$ such that
\begin{equation}\label{modnorm}
\nm f{M^{p,q}_{\omega }} \equiv \Big ( \int \Big ( \int |V_\phi f(x,\xi
)\omega (x,\xi )|^p\, dx\Big )^{q/p} \, d\xi \Big )^{1/q} <\infty
\end{equation}
(with the obvious changes if $p=\infty$ and/or
$q=\infty$). If $p=q$ we simply write $M^p_{\omega }$
instead of $M^{p,p}_{\omega }$, and
if $\omega =1$, then we set $M^{p,q}=M^{p,q}_{\omega }$
and $M^{p}=M^{p}_{\omega}$.

\par

If $\omega $ i $v$-moderate, then the spaces $M_{\omega}^{p,q}$ are (quasi-)Banach spaces
and different choices of  $\phi \in  M^r_{v} \setminus 0$
give rise to equivalent (quasi-)norms in \eqref{modnorm}, and so
$M_{\omega}^{p,q}$ is independent on the choice of
$\phi \in  M^r_{v}$ \cite[Proposition 2.1]{To25}.

For   $p,q\in [1,\infty )$ and $\omega \in\mascP _E(\rr {2d})$ the dual of $ M^{p,q}_{\omega}(\mathbb{R}^d)$ is
$ M^{p',q'}_{1/\omega}(\mathbb{R}^d),$ where $ \frac{1}{p} +  \frac{1}{p'} $ $ =
 \frac{1}{q} +  \frac{1}{q'} $ $ =1.$

\par

For a given weight
$\omega \in\mascP _E(\rr {2d})$ we put $\tilde  \omega $ for the (double) sequence
$\tilde  \omega (n,l)  = \omega (\frac{n}{2},l)$, $ (n,l) \in  \mathbb{Z} \times \mathbb{N}_{0} $.
By $ l^{p,q} _{\tilde  \omega} $, $p,q\in [1,\infty ]$, we denote the space of sequences $ (a_{l,n})_{(l,n) \in  \mathbb{N}_{0}
\times \mathbb{Z}  } $ for which the norm
$$
\| a_{l,n}\|_{l^{p,q} _{\tilde  \omega}} =
\big (\sum_{l=0} ^\infty ( \sum_{n\in \mathbb{Z}}  |a_{l,n}|^p \tilde  \omega (n,l)^p  )^{q/p} \big )^{1/p}
$$
is finite.

\par

The next theorem is analogous to Theorem  \ref{coorbitandwilson}.
It follows from \cite[Chapter 12.3]{Grbook} so we omit the proof.

\begin{thm} \label{modulationandwilson}
Let $p,q\in [1,\infty ]$, $\omega \in\mascP _E(\rr {2d})$,
and let there be given a Wilson basis of exponential decay $ \{ \psi_{l,n} \}_{ l \in \mathbb{N}_{0}, n \in \mathbb{Z}}. $
Then the Banach spaces $M^{p,q}_{\omega }(\rr d)$ and $ l^{p,q} _{\tilde  \omega} $ are isomorphic. An explicit isomorphism
is provided by the coefficient operator $C_\psi: M^{p,q}_{\omega }(\rr d) \rightarrow l^{p,q} _{\tilde  \omega} $ given by
\begin{equation} \label{isomorphism}
C_\psi f =   (\langle f, \psi_{l,n} \rangle )_{(l,n) \in  \mathbb{N}_{0}
\times \mathbb{Z}  }.
\end{equation}
\end{thm}

\par

By Theorem \ref{modulationandwilson} it follows that
$$
f = \sum\limits_{ l \in \mathbb{N}_{0}, n \in \mathbb{Z} } \langle f,  \psi_{l,n} \rangle   \psi_{l,n}
$$
with the unconditional convergence in $ M^{p,q}_{\omega }(\rr d)$ if $1\leq p,q<\infty$, and weak$^*$ convergence
in $ M^\infty _{1/v} (\rr d)$ otherwise.

Gelfand-Shilov spaces and their dual spaces can be described as projective or inductive limits of modulation spaces as follows.

\begin{thm} \label{modproerties}
Let $1 \leq p,q \leq\infty$, $s>1/2$, and set
\begin{equation} \label{eq:weight}
\omega_h (x, \omega) \equiv e^{h (|x|^{1/s} + |\xi|^{1/s})}, \;\;\; h>0, \; x,\xi \in \mathbb{R}^d.
\end{equation}
Then
$$
\Sigma_s (\mathbb{R}^d)=  \bigcap _{h>  0} M _{\omega_{h}} ^{p,q} (\mathbb{R}^d),\;\;\;
(\Sigma_s)' (\mathbb{R}^d) =  \bigcup _{h>0} M _{1/\omega_{h}} ^{p,q} (\mathbb{R}^d),
$$
$$
{\maclS}_s (\mathbb{R}^d)
=  \bigcup _{h >  0} M _{\omega_{h}} ^{p,q} (\mathbb{R}^d),\;\;\;
({\maclS}_s)' (\mathbb{R}^d)
= \bigcap _{h > 0} M _{1/\omega_{h}} ^{p,q} (\mathbb{R}^d).
$$
\end{thm}

\begin{proof}
The proof is  well known, see e.g. \cite[Theorem 3.9]{To11} and  \cite{T3}.

However, we may give a simple independent proof based on Theorem \ref{esemes} when $p=q=2$. Namely, if we put
$ \displaystyle \omega_1 (x) = e^{h|x|^{1/s}} $ and $ \displaystyle \omega_2
(\xi) = e^{h|\xi|^{1/s}}, $ $ x,\xi \in \rr d, $  $ s>1$, and $ h \geq 0 $,
then by definition we have
$$
M^{2} _{1\otimes \omega_2}  (\mathbb{R}^d)=
{\mathcal C}o Y^{h, s}    (\rr d) \quad \text{ and } \quad
M^{2} _{\omega_1 \otimes 1}  (\mathbb{R}^d)=
\mathscr{F} {\mathcal C}o Y^{h, s}    (\rr d),
$$
see also  \cite{PT1}, and the claim follows directly from Theorem \ref{esemes}. The same conclusion for general
$1 \leq p,q \leq\infty$ holds from embedding properties of modulation spaces and certain equivalence properties of norms for Lebesgue spaces. We omit details, and refer the reader to e.g. \cite{T3}.
\end{proof}

As noted by Gr\"ochenig, the isomorphism Theorem \ref{modulationandwilson} can  be formulated in different mathematical language.
For example, in combination with Theorem \ref{modproerties} we conclude that there exists a {\em tame isomorphism} between
FS spaces $\Sigma_s (\mathbb{R}^d)$ and
$ \cap_{h > 0} l^{p,q} _{\tilde \omega_h} $ and  between LS spaces $ {\maclS}_s (\mathbb{R}^d)$ and  $ \cup_{h > 0} l^{p,q} _{\tilde \omega_h} $. We refer to \cite{langen} for  the precise definition of tame isomorphisms and related
considerations in the context of Hermite functions expansions instead of Wilson bases.

Now we can present an alternative proof of our main results.

\begin{proof} {\em (alternative proof of Theorems \ref{main1} and \ref{main2})}
We give the proof for the Roumieu case $ {\maclS}_s (\mathbb{R}^d)$
and $ ({\maclS}_s)' (\mathbb{R}^d)$. The Beurling case can be proved by using similar arguments.

Let $ f \in  {\maclS}_s (\mathbb{R}^d)$. By Theorem \ref{modproerties} it follows that there exists  $h>0$ such that
$ f \in  M _{\omega_{h}} ^{2} (\mathbb{R}^d)$, where $\omega_{h}$ is given by \eqref{eq:weight}.
Now, Theorem \ref{modulationandwilson} implies that
$$
f = \sum\limits_{ l \in \mathbb{N}_{0}, n \in \mathbb{Z} } c_{l,n}    \psi_{l,n}, \quad \text{where} \quad
c_{l,n} = \langle f, \overline{ \psi_{l,n}} \rangle,
$$
and
$ c_{l,n} \in l^{2,2} _{\tilde  \omega _{h}},$  i.e.
$$
 \sum\limits_{ l \in \mathbb{N}_{0}, n \in \mathbb{Z} } |c_{l,n}|^2
e^{2h(|n/2| + |l|)^{1/s}} < \infty, \;\; \mbox{ for some }
\;\; h \geq 0,
$$
which proves Theorem  \ref{main1} a).

The converse part follows from the fact that under the assumption of Theorem
\ref{main1} b) we have
$$
 \sum\limits_{ l \in \mathbb{N}_{0}, n \in \mathbb{Z} } c_{l,n}    \psi_{l,n} \in M^{p,q} _{\omega_h}
$$
for some $h>0$. Therefore, by Theorem \ref{modproerties} it follows that the sum represents a unique element $f \in   {\maclS}_s (\mathbb{R}^d)$, and since the Wilson basis is an ONB it follows that $c_{l,n} = \langle f, \overline{ \psi_{l,n}} \rangle,$
and we are done.

Theorem \ref{main2} follows by duality.
\end{proof}

\begin{rem} We note that the proof of  Theorems \ref{main1} and \ref{main2} given here above does not use any reference to the coorbit space theory.
It relies on the representation of modulation spaces by means of Wilson bases, which follows from the relation between Wilson bases and Gabor frames, cf. \cite{Grbook}. Another ingredient is Theorem \ref{modproerties} which can be proved without the coorbit space theory, see \cite[Theorem 3.9]{To11}.

On the other hand, the proof presented in Section \ref{sec2} is based on direct estimates, and
does not rely on other results, apart from checking whether the construction of Wilson bases fits well
to the general theory of coorbit spaces, which is done in \cite{FGW, PT1}.

We note that in the background of both proofs are decay properties of the STFT and the exponential decay property of the considered Wilson bases. Therefore, the techniques from the present paper can be modified to include other time-frequency representations and
also more general (for example anisotropic) Gelfand-Shilov type spaces, cf. \cite{Te3}.
\end{rem}

\appendix \setcounter{secnumdepth}{0} \section{Appendix: Proof of Lemma \ref{seriesestimate}}\label{appendix}


Since $\sum_{n=1} ^\infty a_n = \infty$, by a small abuse of notation we consider the subsequence so that $ a_n > 0$, $ \forall n \in
\mathbb{N}$.  Consider the partial sum
$ s_{m_0} = \sum_{n=1} ^{m_0}  a_n  = M$, where $ m_0$ is chosen such that
$ a_n < 1 $ for all $ n \geq m_0 - 1.$
Then we have:
$$
M-1 \leq s_{m_0 -1 } < M \leq  s_{m_0} < M+1,
$$
$$
s_{m_0} -  s_{m_0 -1 } = a_{m_0} < 1 \qquad
\text{ and } \qquad
s_{m_0 -1} -  s_{m_0 -2 } = a_{m_0 -1} < 1.
$$
Moreover, $ s_{m_0 + 1} < M+2 $.

Now choose $m_1 \in \mathbb{N} $ as the minimal index such that
$$
s_{m_1} \geq M + 3.
$$
Thus we have $ M+3 >  s_{m_1 - 1} \geq  M +2 $, and $m_1 -1 > m_0$
Moreover, $ s_{m_1} = s_{m_1 - 1} + a_{m_1} < M + 4.$

Therefore we have the following situation:
$$
M \leq s_{m_0} < M+1 < M+2 \leq  s_{m_1 - 1} \leq  M+ 3< s_{m_1} <M +4,
$$
wherefrom
$$
1 < s_{m_1 - 1} - s_{m_0} < 3.
$$

\par

We continue as follows: choose $m_2 \in \mathbb{N}$ such that
$ m_2 - 1 > m_1 $ and
$$
M+3 \leq s_{m_1} < M +4 < M+5 \leq   s_{m_2 - 1} \leq  M+ 6 \leq  s_{m_2} <M + 7.
$$
This gives $ 1 < s_{m_2 - 1} - s_{m_1} < 3.$

By choosing $ m_l $, $ l \geq 3 $, in an analogous way, we obtain an increasing sequence of integers such that
$$
1 < \sum_{n= m_{l-1}} ^{m_l -1} a_n < 3,
$$
which proves the claim.

The same arguments show that for any given
$\varepsilon > 0$ there exists an increasing sequence of integers $ m_l, $ $ l \in  \mathbb{N},$ such that
$$
\frac{\varepsilon}{3} < \sum_{n = m_{l-1}} ^{m_l - 1} a_n  < \varepsilon.
$$

\section*{Acknowledgements}

The work is partially supported  by projects {\em "Localization in Phase space: theoretical, numerical and practical aspects"} No. 19.032/961--103/19 MNRVOID Republic of Srpska, TIFREFUS Project DS 15, and MPNTR of Serbia Grant No. 451--03--9/2021--14/200125.

\par


\end{document}